\documentclass{birkjour}
\usepackage{proof}
\usepackage{multirow}
 \newtheorem{thm}{Theorem}[section]
 \newtheorem{cor}[thm]{Corollary}
 \newtheorem{lem}[thm]{Lemma}
 \newtheorem{prop}[thm]{Proposition}
 \theoremstyle{definition}
 \newtheorem{defn}[thm]{Definition}
 \theoremstyle{remark}
 
 \newtheorem*{ex}{Example}
 \numberwithin{equation}{section}

\begin{document}
%
%
%
%
%
%
%
%
%

\title[Resolution for CPPL]
 {Resolution for Constrained Pseudo-Propositional Logic}

\author[Ahmad-Saher Azizi-Sultan]{Ahmad-Saher Azizi-Sultan}

\address{P.O. Box 30098\\
Taibah University\\
Medina Munawarah\\
Saudi Arabia}

\email{sultansaher@hotmail.com}




\date{   }



\begin{abstract}
    This work, shows how propositional resolution can be generalized to obtain a resolution proof system for constrained pseudo-propositional logic (CPPL), which is an extension resulted from inserting the natural numbers with few constraints symbols into the alphabet of propositional logic and adjusting the underling language accordingly. Unlike the construction of CNF formulas which are restricted to a finite set of clauses, the extended CPPL does not require the corresponding set to be finite. 
    Although this restriction is made dispensable, this work presents a constructive proof showing that the generalized resolution for CPPL is sound and complete. As a marginal result, this implies that propositional resolution is also sound and complete for formulas with even infinite set of clauses.
\end{abstract}

\maketitle

\section{Introduction}
The gradual improvements of SAT algorithms, developed over the last few decades,  have made them the ideal choice for approaching a wide range of important problems such as formal verification \cite{Biere:1999:SMC:309847.309942,Biere:1999:SMC:646483.691738,VELEV200373}, planning \cite{Kautz:1996:PEP:1864519.1864564, DBLP:books/lib/RussellN03}, scheduling \cite{DBLP:conf/aips/GomesSMT98} and so forth.
Generally, these algorithms are based on the resolution proof system, which is built within the framework of propositional logic where formulas are represented by their CNF. 

Naturally many important problems in different areas contain counting constraints.
Unfortunately the limitations of the expressive power of propositional language does not allow for counting tools to be built within the framework of the language itself.
Thus, in general, encoding counting constraints in CNF increases the number of  clauses and variables extensively \cite{DBLP:conf/sara/AavaniMT13, DBLP:conf/sat/BailleuxBR09}, resulting in impractical excessive complexity as  SAT is one of the canonical NP-complete problems \cite{DBLP:conf/stoc/Cook71}. This had motivated the development of CPPL, which evolved from propositional logic by interpolating the natural numbers with few constraint symbols into its alphabet and adjusting the underlying language accordingly \cite{DBLP:conf/cade/Azizi-Sultan18,DBLP:journals/lu/Azizi-Sultan20}.
The resulting evolved logic became capable of formulating counting constraints or SAT instances naturally and much more succinctly than having them being encoded using CNF. 
Moreover, its evolved proof system retained its soundness and completeness \cite{DBLP:journals/lu/Azizi-Sultan20}.
The inference rules involved are  listed in Table \ref{IRs}. 
From this table, It can be noticed that the list of inference rules is quite long and hence unfortunate for computerized implementation.
A proof system with fewer inference rules results in less cases to be considered and enhances its complexity.

Analogously to propositional resolution, this work is presenting a resolution proof system for CPPL, with only two rules involved, and showing that it is sound and complete.

\section{Background}
Consider the set of natural numbers, $\mathbb{N}$,  and let $\mathcal{P} = \{ p,p_1,p_2,\dots \}$ be a finite or countably infinite set of propositional variables.

\subsection{PPL}
The {\em alphabet}  $\mathcal A$ underling the language of PPL is the union $\mathbb{N} \cup \mathcal{P}  \cup  \left\lbrace  \neg, +, \left( , \right) \right\rbrace$,
where the symbols $\neg$, $+$, (, and ) resemble the negation, addition, opening and closing punctuation symbols, respectively.

The language of PPL is the set of  {\em formulas} $\mathcal{F}$  which is defined to be  the smallest set satisfying the two rules:
\begin{enumerate}
    \item[] (r1) $\mathcal{P} \subset \mathcal{F}$.
    \item[] (r2) If $n \in \mathbb{N}$ and the strings $\alpha$ and $\beta$ are formulas, then the strings $n (\alpha)$, $\neg \alpha$ , and $(\alpha + \beta)$  are also formulas.
\end{enumerate}

For the sake of simplicity and smoothing readability, the following additional definitions and notations are adopted:
\begin{enumerate}
    \item[$\bullet$]  Every formula which is a propositional variable is called {\em prime formula}.
    \item[$\bullet$] As in writing arithmetical terms, parenthesis are omitted whenever it is possible.
    \item[$\bullet$] A prime formula or its negation is called a \emph{literal} and denoted by $\lambda$.
    \item[$\bullet$] A formula of the form $n(\lambda)$, or simply $n \lambda$, is called \emph{pseudo-literal}.
    \item[$\bullet$] A formula which is a pseudo-literal or an addition of two or more pseudo-literals is said to be in \emph{normal form}.  
    \item[$\bullet$] A \emph{subformula} of a formula $\alpha$ is a substring occurring in $\alpha$ that is itself a formula. 
    \item[$\bullet$] Formulas will be denoted by $\alpha, \beta, \dots$, and sets of formulas by $F$ and $F'$, where these symbols may also be indexed. 
\end{enumerate}

Having defined formulas, it is time now to look at their meanings and semantics. 
A proposition can only have one of the \emph{truth values}, true or false.
In PPL, these values are represented by $(1,0)$ for true and $(0,1)$ for false.
Each of which can serve as a meaning for a prime formula. 
Furthermore, assigning a meaning to some formula $F \in \mathcal{F}$, is done by the aid of the following three functions:
\begin{itemize}
    \item Scalar multiplication: $\mathbb{N} \times \mathbb{Z}^2 \rightarrow \mathbb{Z}^2$ where $n(z,z') = (nz,nz') $.
    \item Negation $\neg^*:\mathbb{Z}^2 \rightarrow\mathbb{Z}^2$ where $\neg^*(z,z') =(z',z)$. 
    \item Addition\footnote{This addition is easily distinguished from the addition of formulas.} 
    $+:\mathbb{Z}^2 \times\mathbb{Z}^2 \rightarrow \mathbb{Z}^2$, where $+((n, m),(k, l)) = (n+k, m+l)$.
\end{itemize}

Yet, an {\em interpretation} $I$ is a subset of $\mathcal{P}$ represented by the mapping $I : \mathcal{F} \rightarrow \mathbb{Z}^2$ which is defined recursively as follows:   
\begin{itemize}
    \item[1)] {\bf Recursion base.} If $\alpha$ is a prime formula then $\alpha = p$ for some $p \in \mathcal{P}$. 
    The recursion base for this case reads 
    $$
    I(\alpha) = I(p) = 
    \begin{cases}
    \,(1,0) &\mbox{ if } p \in I,\\
    \,(0,1) &\mbox{ if } p \notin I.\\
    \end{cases}.
    $$
    \item[2)] {\bf Recursion steps.}\\
    $$
    I(\alpha) =
    \begin{cases}
    n\,(I(\beta))     &\text{ if $\alpha$ is of the form } n \, \beta, \mbox{ where } n \in \mathbb{N}, \\
    \neg^* (I(\beta)) &\mbox{ if $\alpha$ is of the form } \neg \beta,\\
    I(\beta_1)  + I(\beta_2) &\mbox{ if $\alpha$ is of the form } \beta_1 + \beta_2.\\
    \end{cases}
    $$
\end{itemize}

If two formulas have the same meaning they are called equivalent. 
More precisely, it is said that $\alpha$ and $\beta$ are \emph{equivalent}, in symbols $\alpha \equiv \beta$, iff for every interpretation $I$ we have $I(\alpha) = I(\beta)$.

\subsection{CPPL}
Given the set of constraint types $\mathcal{C} =  \{ >,\geq, = , <,\leq\}$, CPPL utilizes these constraint types additionally to the PPL formulas to build up its language.
The {\em{alphabet}} $\mathcal{A}_c$ underlying the language of CPPL is formed by the unions 
$\mathcal{F}  \cup  \mathcal{C}   \cup 
\left\lbrace  \sim, \left( , \; \right)\right\rbrace \cup \mathbb{N}$,
where $\sim$, (, and ) resemble the negation and opening and closing punctuation symbols, respectively.  

The set of {\em sentences} in CPPL, denoted $\mathcal{S}$, is the smallest set  satisfying the two rules: 
\begin{enumerate}
    \item[] (R1) Every triple $(\varphi, \bowtie, n) \in \mathcal{F} \times \mathcal{C} \times \mathbb{N}$  forms a sentence, denoted $\varphi^{(\bowtie,n)}$. 
    \item[] (R2) If the string $\Phi$ is a sentence, then the string $\sim (\Phi)$ is also a sentence.
\end{enumerate}

Sentences are denoted by the Greek letters A, B, $\Gamma$, \dots, and sets of sentences are indicated by $S$ and $S'$, where these symbols may also be indexed.
Moreover, parenthesis are omitted whenever it is possible.  Hence  $\sim (\Phi)$ is simply written as $\sim \Phi$.

It is said that an interpretation $I$ is a {\em model} for a sentence $\Phi = \varphi^{(\bowtie,n)}$, in symbols $I \models \Phi$ or $I \models \varphi^{(\bowtie,n)}$, iff $I(\varphi) = (m,l) $ with $m \bowtie n$.
It is also said $I$ is a model for $S$ or $I \models S$, iff $I \models \Phi$ for every sentence $\Phi \in S$.

A  direct consequence of this definition is the following proposition.

\begin{prop}
    \label{prop:sum_qs}
    Let $\varphi^{(\bowtie,\,n)}, \psi^{(\bowtie,\, m)} \in \mathcal{S}$. 
    If $I$ is an interpretation such that $I \models \varphi^{(\bowtie,\,n)}$ and $I \models \psi^{(\bowtie,\,m)}$, then $I \models (\varphi + \psi)^{(\bowtie,\,n+m)}$. 
\end{prop}

To this end, satisfiability and semantic equivalence in CPPL can be considered analogously to propositional logic.
\begin{itemize} 
    \item It is said that $\Phi$ (resp. $S$) is \emph{satisfiable} iff there exists an interpretation $I$ such that $I \models \Phi$ (resp. $I \models$ $S$).
    \item It is said that $\Phi$ (resp. $S$) is \emph{valid} or a \emph{tautology} iff for all interpretations $I$ we find that $I \models \Phi$ (resp. $I \models S$).
    \item It is said that $\Phi$ (resp. $S$) is \emph{unsatisfiable} iff for every interpretation $I$ we have $I \not \models \Phi$ (resp. $I \not \models$ $S$).  If $\Phi$ is unsatisfiable it will be abbreviated by $\bot$.
    \item  It is said that $S'$ is a {\em logical consequence} of $S$, denoted $S$ $\models$ $S'$, if $I \models$ $S$ $\Rightarrow I \models$ $S'$ for all interpretations $I$. 
    \item  It is said that $S$ and $S'$ are {\em semantically equivalent}, denote $S \equiv S'$,  if $S$ $\models$ $S'$ and $S'$ $\models$ $S$ 
\end{itemize}

Conventionally, $S \models \Phi$ (resp. $\Phi \models$ $S$) means $S$ $\models \{\Phi\}$ (resp. $\{\Phi\} \models S$).
More generally, $S \models\Phi_1,\Phi_2, \dots,\Phi_k$ (resp. $\Phi_1,\Phi_2, \dots,\Phi_k \models$  $S$) is written instead of  $S$ $\models \{\Phi_1,\Phi_2, \dots,\Phi_k\}$ (resp. \{$\Phi_1,\Phi_2, \dots,\Phi_k\} \models S$), and more briefly $S,\Phi \models \Psi$ is written instead of $S \cup \{ \Phi \} \models \Psi$.
Analogous notations are used regarding semantic equivalence.

According to these conventions, defining negation in CPPL becomes straightforward.
Let  $\Phi , \Psi \in \mathcal{S}$. If for any interpretation $I$ we have $I \models \Phi  \iff  I \not \models \Psi$ and $I \not \models\Phi \iff  I \models \Psi$, then it is said that $\Psi$ is the {\em negation}  of $\Phi$ and abbreviated  by $\sim \Phi = \Psi$. 
Thus,  $I \models \Phi \iff I \not  \models \sim \Phi \iff I  \models \sim \sim \Phi$. That is $\Phi \equiv \sim \sim \Phi$.

Under the presented terminology, the following theorems can be proven (see \cite{DBLP:journals/lu/Azizi-Sultan20}). 

\begin{thm}[Resolution theorem]
    \label{theo:Resolution}
    If $\varphi \in \mathcal F$ and $n > m' \geq m$, then the following consequence holds:
    $$(\varphi + m' p + m(\neg p))^{(\bowtie,\, n)} \models (\varphi+(m'-m)p)^{(\bowtie,\, n-m)}.$$
\end{thm}

\begin{thm}
    \label{theo:negation}
    If $ \Phi = \varphi^{(\geq,\, n)}$ is a sentence and $N$ is the coefficient sum of $\varphi$, then the negation of $\Phi$ can be represented by the sentence $(\neg \varphi)^{(\geq,\, N -n+1)}$. That is $\sim \Phi = (\neg \varphi)^{(\geq,\, N -n+1)}$.	
\end{thm}

Finally, to conclude this section, one more definition is needed.
It is said that the sentence $\varphi^{(\bowtie,\,n)}$ is in {\em standard form} or simply {\em standard}, if the formula $\varphi$ is in its normal form and the symbol ``$\bowtie$" corresponds to the constraint type ``$\geq$". 
$S$ is in standard form if every sentence $\Phi \in S$ is in its standard form.

Every sentence in CPPL can be transformed equivalently into a standard sentence \cite{DBLP:journals/lu/Azizi-Sultan20}.
Thus any $\Phi \in \mathcal{S}$, can be represented as follows: 
$$
\Phi = \left(  \sum_{i = 1}^{i = k}  n_i \lambda_i \right)^{(\geq,\, n)},
$$

The ultimate goal of CPPL, as in every logic,  is to have a system of formal proof which consists of inference rules for deducing the logical consequence of sentences from that of another. 
Restricting inference rules to standard sentences results in giving fewer cases to be handled, and hence simplifies the proof system and enhances its implementational efficiency. 
Thus, the indicated restriction will be utilized throughout the remaining part of this paper.

\subsection{Formal proofs}
A {\em proof system} consists of a set of inference rules that allow for deducing sentences from another set of sentences. 
If $S'$ can be deduced from $S$ by one or several applications of the inference rules, then it is said that $S'$ is {\em derivable} from $S$ and denoted by $S \vdash S'$. 

Each inference rule consists of certain premises that lead to a certain conclusion and written as $\frac{Premises}{Conclusion}$.
Thus the inference rule  $\frac{S \subseteq S'}{S' \vdash S}$, called the initial rule, means that $S$ is derivable from $S'$ whenever $S \subseteq S'$.

 Given a proof system, such as the one listed in Table~\ref{IRs},
a {\em formal proof} is a finite sequence of statements of the form $S \vdash S'$.
Each of these statements is a result from the previous ones and  an  application of some inference rule in the system.
If there is a formal proof that ends up with the statement $S \vdash S'$, it is said that $S'$ is {\em derivable} from $S$.
Additionally, if whenever $S \vdash S'$ it turns out that  $S \models S'$ then it is said that the proof system is {\em sound}.
If, on the other hand, we can prove  that $S \models S'$ implies $S\vdash S'$, then it is said that 
the proof system is  {\em complete}. For example, the proof system listed in Table~\ref{IRs} is sound and complete \cite{DBLP:journals/lu/Azizi-Sultan20}.

\begin{table}[!h]
    \begin{center}
        \begin{tabular}{|r | l|}
            \hline
            \vspace{-3 mm} & \\
            \infer{\text{Conclusion}}{\text{Premices}}   & \raisebox{0.5 em}{Name}\\
            \hline \hline
            \vspace{-3 mm} & \\
            \infer{S' \vdash S}{S \subseteq S'} &  \raisebox{0.8 em}{Initial Rule} \\
            \vspace{-3 mm} & \\ 
            \infer{S_2\vdash S}{S_1 \vdash S & S_1\subseteq S_2} &  \raisebox{0.8 em}{Rule of Monotonicity I } \\ 
            \vspace{-3 mm}  & \\
            \infer{S\vdash S_2}{S \vdash S_1 & S_2\subseteq S_1} &  \raisebox{0.8 em}{Rule of Monotonicity II } \\ 
            \vspace{-3 mm}             & \\
            \infer{S\vdash S_1 \cup S_2}{S\vdash S_1 & S \vdash S_2} &  \raisebox{0.8 em}{Union Rule} \\ 
            \vspace{-3 mm}         & \\
            \infer{S\vdash S'}{S , \Phi \vdash S' & S, \sim \Phi \vdash S'} &  \raisebox{0.8 em}{Negation Rule} \\ 
            \vspace{-3 mm}  & \\
            \infer{S' \vdash S}{S' \vdash \bot} &  \raisebox{0.8 em}{Constant Rule} \\ 
            \vspace{-3 mm}             & \\
            \infer{S \vdash (\varphi + \psi)^{(\geq,\,m+n)}}{S \vdash \varphi^{(\geq, m)} ,  \psi^{(\geq,\,n)}} &  \raisebox{0.8 em}{Addition Rule} \\ 
            \vspace{-3 mm}            & \\
            \infer{S \vdash (\varphi+(m'-m)p)^{(\geq,\, n-m)}}{n > m' \geq m & S \vdash (\varphi + m' p + m(\neg p))^{(\geq,\, n)} } &  \raisebox{0.8 em}{Resolution Rule }\\
            \hline
        \end{tabular}
        \caption{ Rules for derivations.}
        \label{IRs}
    \end{center}
\end{table}

If we have a proof system that is sound, then any formal proof for $S \vdash S'$ provides an alternative to the mathematical thinking for determining whether $S'$ is a consequence of $S$.
Instead of thinking, we could simply follow a set of rules ending up with the desired assertion.

\begin{ex}
    Considering the proof system listed in Table \ref{IRs},  a formal proof for the implication $S,\sim \Phi  \vdash \bot \implies S \vdash \Phi$ can be represented in a three-column table, like Table \ref{Ta_Lemma_C_Back}.
    \begin{table}[!h]
        \begin{center}
            \begin{tabular}{r l l}
                \hline
                Line & Statement  & Justification \\
                \hline
                1.& $S,\sim \Phi  \vdash \bot$  &  Assumption \\
                2.& $S,\sim \Phi  \vdash \Phi$  &  Constant Rule applied to line 1 \\
                3.& $S,  \Phi  \vdash \Phi$  &  Initial Rule \\
                4.& $S \vdash \Phi$  &  Negation Rule applied to lines 2 and 3 \\
                \hline
            \end{tabular}
        \end{center}
        \caption{Formal proof for $S,\sim \Phi  \vdash \bot$ implies $S \vdash \Phi$.}
        \label{Ta_Lemma_C_Back}
    \end{table}
\end{ex}

Although this proof system enables us to prove assertions by applying blindly a set of rules, it  is intended for human use. 
A proof system is ideal, if it could completely replace the thought process so that we could program a computer to perform the needed proofs.

The rest of this paper is devoted to present a proof system with minimal number of rules called CPPL resolution, analogously to  propositional resolution. 
It takes the thinking process out of formal proofs allowing for computerized deductions. 

\section{ CPPL resolution }
CPPL resolution is a system of formal proof with the only two rules listed in Table \ref{RIRs}.
 \begin{table}[!b]
     \begin{center}
         \begin{tabular}{|r | l |}
             \hline
             \vspace{-3 mm} & \\
             \infer{\text{Conclusion}}{\text{Premices}}   & \raisebox{0.5 em}{Name}\\
             \hline \hline
             \vspace{-3 mm} & \\
             \infer{S' \vdash^r S}{S \subseteq S'} &  \raisebox{0.8 em}{Rule I} \\
             \vspace{-3 mm}            & \\
             \infer{S \vdash^r (\varphi_1 + \varphi_2)^{(\geq, n_1+n_2 - |\beta|)}}{S \vdash^r (\varphi_1 + \beta)^{(\geq, n_1)}, \, (\varphi_2 + \neg \beta)^{(\geq, n_2)}  } &  \raisebox{0.8 em}{Rule II}\\
             \hline
         \end{tabular}
         \caption{Derivation rules of CPPL resolution.}
         \label{RIRs}
     \end{center}
 \end{table}

Please note that, the derivation symbol in CPPL resolution proof system is symbolized by $\vdash^r$ to differentiate it from the derivation proof system listed in Table \ref{IRs}. 

Our goal now is to prove that the CPPL resolution is sound and complete.
Let us start with the soundness part, which is rather straightforward.

\begin{thm}[Soundness]
    \label{theo_sound}
    If $S \vdash^r S'$, then  $S \models S'$.
\end{thm}

\begin{proof}
    Assuming $S \vdash^r S'$ means that there is a formal proof of several lines, each contains a statement of the form $X \vdash^r Y$ justified by one of the two rules of Table \ref{RIRs}, and the final line contains and justifies the statement $S \vdash^r S'$.
    Thus proving this theorem  is achieved by showing that for each line of the formal proof, if $X \vdash^r Y$, then $X \models Y$.
    But this follows from verifying the two inference rules in Table \ref{RIRs}.
    
    Rule I concludes with ${S' \vdash^r S}$ under the premise $S \subseteq S'$.
        Thus we need to show that if $S \subseteq S'$, then ${S' \models S}$. 
    This verification is rather easy. 
    ${I \models S'}$ means ${I \models \Phi'}$ for all $ \Phi' \in  S'$. 
    Since $S \subseteq S'$, it follows that ${I \models \Phi}$ for all $ \Phi \in  S$. 
    Thus $S' \models S$.
    
    To verify Rule II  we need to show that the consequence  
    $S \models (\varphi_1 + \beta)^{(\geq, n_1)}, \, (\varphi_2 + \neg \beta)^{(\geq, n_2)} $ implies
    $S \models (\varphi_1 + \varphi_2)^{(\geq, n_1+n_2 - |\beta|)}$.
    Considering the standard form, we can assume that $\beta = \sum_{i=1}^{i=k} m_i\lambda_i$.  Accordingly, this means that $\neg \beta = \sum_{i=1}^{i=k} m_i \neg \lambda_i$.
    Now, let  $I \models S$. From the assumption this means that 
    $ I \models (\varphi_1 + \beta)^{(\geq, n_1)}$ and $ I \models (\varphi_2 + \neg \beta)^{(\geq, n_2)}$. 
    Taking into account Proposition \ref{prop:sum_qs}, we conclude that 
    $ I \models (\varphi_1 + \beta + \varphi_2 + \neg \beta)^{(\geq, n_1+n_2)}$.
    That is  
    $ I \models (\varphi_1 + \sum_{i=1}^{i=k} m_i\lambda_i  + \varphi_2 + \sum_{i=1}^{i=k} m_i \neg \lambda_i  )^{(\geq, n_1+n_2)}$. 
    Applying Theorem \ref{theo:Resolution} $k$ times results in $I \models (\varphi_1 + \varphi_2)^{(\geq, n_1+n_2 - \sum_{i=1}^{i=k} m_i)}$.
    Hence, $I \models (\varphi_1 + \varphi_2)^{(\geq, n_1+n_2 - |\beta|)}$. 
    We conclude that $ S\models (\varphi_1 + \varphi_2)^{(\geq, n_1+n_2 - |\beta|)}$.   
\end{proof}

The soundness theorem means that $\vdash^r \; \subseteq \; \models$. 
In order to verify the opposite inclusion, which presents the completeness theorem, some more notations and results are needed.

\begin{lem}
    \label{lem_mono}
    For any $S' \subseteq S$, if $ S \not \models \Phi$ then  $ S' \not \models \Phi$.
\end{lem}
\begin{proof}
    $ S \not \models \Phi$ implies that there is a model $I$ such that $I \models S$ but $ I \not \models \Phi$. Since $S' \subseteq S$ it follows that $I \models S'$ and $ I \not \models \Phi$. 
    That is $S' \not \models \Phi$. 
\end{proof}

Taking into account the soundness theorem, which also means $\not \models  \subseteq \not \vdash^r $, Lemma \ref{lem_mono} entails the following corollary.

\begin{cor}
    \label{cor_mono}
    For any $S' \subseteq S$, if $ S \not \vdash^r \Phi$ then  $ S' \not \vdash^r \Phi$.
\end{cor}
\begin{defn}
    \label{Def_DeductClosure}
    Let $S$ be a set of sentences. It is said that $S$  is  {\em deductively closed} if it contains every sentence  that is derivable from it. Formally:  if $ S \vdash^r S'$ then $S' \subseteq S$. 
    The {\em deductive closure } of $S$, denoted $D(S)$  is the  smallest set that is deductively closed and contains $S$.    
\end{defn}

\begin{cor}
    \label{cor_Monotonicity}
    If $S' \subseteq S $ then $ D(S') \subseteq D(S) $.
\end{cor}
\begin{proof}
   Assume that $D(S') \not \subseteq D(S)$. 
   That is there exists $\Phi \in D(S')$ such that $\Phi \not \in D(S)$. 
   This implies that $S \not \vdash^r \Phi$.
   Since $S' \subseteq S $ from Corollary \ref{cor_mono} we conclude that $ S' \not \vdash^r \Phi$.
   This means that $\Phi \not \in D(S')$ which results in a contradiction.
\end{proof}

\begin{cor}[Union Rule]
    \label{cor_Union}
    If ${S \vdash^r S_1} $  and ${S \vdash^r S_2}$, then  $ {S \vdash^r S_1 \cup  S_2}$. 
\end{cor}
\begin{proof}
    ${S \vdash^r S_1} \Rightarrow S_1 \subseteq D(S)$ and ${S \vdash^r S_2} \Rightarrow S_2 \subseteq D(S)$.   
    It follows that ${S_1 \cup S_2 \subseteq D(S)}$, which implies the derivation of $S \vdash^r S_1 \cup S_2$.
\end{proof}

\begin{lem}[Constant Rule]
    \label{lem_ConstantRule} 
    If  $S \vdash^r \bot$, then  $S \vdash^r \Phi  $.
\end{lem}
\begin{proof}
    Suppose $S \not \vdash^r \Phi $. It follows that $S \not \vdash^r \Phi, \sim \Phi$ or $S \not \vdash^r \bot$ which contradict the assumption of $S \vdash^r \bot$.
\end{proof}

\begin{lem}[Contradiction rule]
    \label{lem_ContradictionRule}
    If $ S \vdash^r \Phi, \sim \Phi $, then $S \vdash^r \bot$.
\end{lem}
\begin{proof}
    Suppose that $\Phi = \varphi^{(\geq,\, n)}$.
    Thus from theorem \ref{theo:negation} it follows that $\sim~\Phi = (\neg \varphi)^{(\geq,\,|\varphi|-n+1)}$. 
    Yet, a formal proof for this example can be represented in a three-column table, namely Table~\ref{Ta_prof_contrRule}.
    \begin{table}[!h]
        \begin{center}
            \begin{tabular}{r l l}
                \hline
                Line & Statement  & Justification\\
                \hline
                1.& $S \vdash^r \varphi^{(\geq,\, n)}, (\neg \varphi)^{(\geq,\,|\varphi|-n+1)}$ &  Assumption \\ 
                2.& $S \vdash^r (\,)^{(\geq,\, n + |\varphi|-n+1 - |\varphi|)} = ( \, )^{(\geq,\, 1)} = \bot$ & Rule II applied to line 1 \\
                \hline
            \end{tabular}
        \end{center}
        \caption{ Formal proof for the contradiction rule.}
        \label{Ta_prof_contrRule}
    \end{table}
\end{proof}

\begin{defn}
    It is said that $S$ is {\em inconsistent} if $S \vdash^r \bot$ and otherwise {\em consistent}.
    $S$ is called {\em maximally consistent} if $S$ is consistent but each $S' \supset S$ is inconsistent.
\end{defn}

In the context of consistency the following corollary helps in proving  important lemmas.
\begin{cor}
    \label{cor_negation}
    If $S \models \sim \Phi $,  then $S  \not \models \Phi $.
\end{cor}
\begin{proof}
    Let  $I \models S $. Since $S \models \sim \Phi $ it follows that $I \models \sim \Phi$.
    This implies  that   $I \not \models \Phi $. 
    We conclude that $S  \not \models \Phi $.
\end{proof}

\begin{lem}
    \label{lem_essen}
    Let $S$ be a consistent set such that  $S \not  \vdash^r \Phi$ and $S  \not  \vdash^r \sim \Phi$ then   $S \cup \{\Phi\}$ is also consistent.
\end{lem}
\begin{proof}
    Assume that $S \cup \{\Phi\}$ is inconsistent. 
    This implies the derivation  $S, \Phi  \vdash^r \bot$. 
    From Lemma \ref{lem_ConstantRule} it follows that $S, \Phi \vdash^r \sim \Phi$. 
   By Theorem \ref{theo_sound}, this implies the consequence  $S, \Phi \models \sim \Phi$.
   Applying corollary \ref{cor_negation},  we conclude that $S, \Phi \not \models \Phi$ which is a contradiction.
 \end{proof}

\begin{lem}
    \label{lem_C}
    $S \not \vdash^r \Phi$ iff $S, \sim \Phi \not \vdash^r \bot$.
\end{lem}
\begin{proof} 
    Let  
            \begin{equation}
    \label{lem_C_eq1}
    S \not \vdash^r \Phi
    \end{equation}
     and suppose that 
        \begin{equation}
    \label{lem_C_eq2}
    S, \sim \Phi  \vdash^r \bot.
    \end{equation}
    By Lemma \ref{lem_ConstantRule}, this implies $S, \sim \Phi  \vdash^r \Phi$.
    From Theorem \ref{theo_sound}, it follows that $S, \sim \Phi  \models \Phi$ which implies $S \models \Phi$. By corollary \ref{cor_negation}, that is $S \not \models \sim \Phi$.
    From Theorem \ref{theo_sound} we conclude that 
    \begin{equation}
    \label{lem_C_eq3}
    S \not \vdash^r \sim \Phi.
    \end{equation}
    Applying Lemma \ref{lem_essen} to (\ref{lem_C_eq1}) and (\ref{lem_C_eq3}) gives the consistency of the set  $S \cup \{\sim \Phi\}$. That is $S , \sim \Phi  \not \vdash^r  \bot$, which contradicts the derivation of (\ref{lem_C_eq2}).

    To prove the opposite direction let  
    \begin{equation}
    \label{lem_C_eq4}
    S, \sim \Phi \not \vdash^r \bot
    \end{equation}
    and assume that $ S  \vdash^r  \Phi.$ 
    This means that 
        \begin{equation}
    \label{lem_C_eq5}
   \Phi \in D(S) \subseteq D(S \cup \{\sim \Phi\})
    \end{equation}
    On the other hand, $S, \sim \Phi  \vdash^r  \sim \Phi$. Thus
            \begin{equation}
    \label{lem_C_eq6}
   \sim  \Phi \in  D(S \cup \{\sim \Phi\})
    \end{equation}
    From (\ref{lem_C_eq5}) and (\ref{lem_C_eq6}) we conclude that $ S, \sim \Phi  \vdash^r \Phi , \sim \Phi$ or $ S, \sim \Phi  \vdash^r \bot$ which contradicts (\ref{lem_C_eq4}). 
\end{proof}

Consequently, $S \not  \vdash^r \Phi$ is equivalent to the consistency of $S':= S \cup \{\sim \Phi\}$, and $S \not  \models \Phi$ is equivalent to the satisfiability of $S'$.
Thus, to prove the claim  $\models \; \subseteq \; \vdash^r $, which is equivalent to $S \not  \vdash^r \Phi \Rightarrow S \not \models \Phi$, we only need to show that consistent sets of sentences are satisfiable. 
In order to accomplish this we need the following lemma,  whose proof can be achieved as a consequence  of Zorns's lemma\footnote{ {\bf Zorns's lemma} If every chain in a nonempty partially ordered set $S$ is bounded then $S$ has a maximal element.}, which is one of the basic mathematical tools \cite{Lan93}.

\begin{lem}
    \label{lem_consExten}
    Every consistent set of sentences $S \subseteq \mathcal{S}$ can be extended to a maximally consistent set $S' \supseteq S$.
\end{lem}

\begin{lem}
    \label{lem_MCS}
    A maximally consistent set $S \subseteq \mathcal{S}$ is satisfiable.
\end{lem}

\begin{proof}
    Assume that $S$ is unsatisfiable. 
    It follows that  for any interpretation $I$ there exists a sentence $\Phi \in S$ such that $I  \not \models \Phi$. 
    Thus $I  \models \sim \Phi$ To this end, we have two cases.
    Case 1: $\sim \Phi \in S $. It follows from the initial rule that $S \vdash^r \sim \Phi$.
    Since $\Phi \in S$ it also follows that $S \vdash^r \Phi$. 
    Applying Corollary \ref{cor_Union} on $ S \vdash^r \sim \Phi $ and  $S \vdash^r \Phi$ results in $S\vdash^r \Phi, \sim \Phi$. 
    By Lemma \ref{lem_ContradictionRule} this implies that $S \vdash^r \bot$, which contradict the consistency of $S$.
    Case 2: $\sim \Phi \not \in S $. This means that $ S \subset S':= S \cup \{\sim \Phi\} $. 
    From $\sim \Phi \not \models \Phi$ it follows that $S' \not \models \Phi$. This implies that $S' \not \vdash^r \Phi$.
    By Lemma \ref{lem_C} this is equivalent to the consistency of $S' \cup \{\sim\Phi \} = S'$, which contradicts the maximality of $S$. 
\end{proof}

Lemmas \ref{lem_consExten} and \ref{lem_MCS} facilitate our way to obtain the most significant result of this work, namely the completeness theorem.
\begin{thm}[Completeness theorem]
    For any $S, S' \in \mathcal{S}$, if $S \models S' $ then $  S \vdash^r S'$. 
\end{thm}  
\begin{proof}
      Let us assume that $S \not \vdash^r S'$. It means that there exists $\Phi \in  S'$ such that  $S \not \vdash^r \Phi$.
      But this implies the consistency of $S \cup \{\sim \Phi \}$ that can be, according to Lemma \ref{lem_consExten}, extended to a maximally consistent set of sentences, say  $S''$. 
      By Lemma \ref{lem_MCS}, $S''$ is satisfiable, hence $S \cup \{\sim \Phi \}$ is also satisfiable. 
      Thus $S \not \models \Phi$. And since $\Phi \in S'$ it follows that $S \not \models S'$. 
\end{proof}

\section{Conclusion}
In this work, the well-known proof system of propositional resolution has been utilized for CPPL.
The resultant CPPL resolution can be viewed as a generalization of propositional resolution. 
This is justified by the fact that every CNF formula, which is a finite set of clauses, can be represented equivalently by a set of sentences in CPPL \cite{DBLP:conf/cade/Azizi-Sultan18}.
Additionally, CPPL does not require the corresponding set of sentences to be finite. 
Although such a restriction is made dispensable, this work has provided us with a constructive proof showing that the CPPL resolution is sound and complete. One can conclude, as a marginal consequence of this result,  that propositional resolution is maintained to be sound and complete even for formulas with infinite set of clauses.
\subsection*{Acknowledgment}
I would like to thank Prof. Steffen H\"olldobler, TU Dresden, Germany. I have learned from him how to rationalize logically rather than just thinking mathematically.  Without his influence this work would not have been initiated.

\bibliographystyle{spmpsci}      
\bibliography{reference}   
\end{document}